\documentclass{amsart}
\usepackage{amsthm}
\usepackage{amssymb}
\usepackage{amsmath}

\usepackage{mathdots}
\newtheorem{theo}{Theorem}[section]
\newtheorem{prop}[theo]{Proposition}
\newtheorem{cor}[theo]{Corollary}
\newtheorem{lm}[theo]{Lemma}

\theoremstyle{definition}
\newtheorem{ex}[theo]{Example}

\newcommand{\la}{\lambda}
\newcommand{\La}{\Lambda}
\newcommand{\Z}{\mathbb{Z}}
\newcommand{\we}{\wedge}

\DeclareMathOperator{\Hom}{Hom}
\DeclareMathOperator{\Ext}{Ext}

\title{On homomorphisms involving a hook Weyl module}	
\author{Mihalis Maliakas}
\address{Department of Mathematics, University of Athens} 
\email{mmaliak@math.uoa.gr}

\author{Dimitra-Dionysia Stergiopoulou}
\address{Department of Mathematics, University of Athens}
\email{dstergiop@math.uoa.gr}

\begin{document}

\maketitle
\begin{abstract}
	For an infinite field $k$ of positive characteristic, we determine all homomorphisms between Weyl modules for $GL_n(k)$, where one of the partitions is a hook. As a consequence we obtain a non-vanishing result concerning homomorphisms between Weyl modules for algebraic groups of type $ B $, $ C $ and $ D$ when one of the partitions is a hook. 
\end{abstract}
\section{Introduction}
Let $k$ be an infinite field of characteristic $p>0$. In the polynomial representation theory of the general linear group $GL_n(k)$, the Weyl modules $\Delta(\lambda)$, indexed by partitions $\lambda$, play an important role and determining their structure is a major problem in the area. In particular, the spaces $\Hom_{GL_n{(k)}}(\Delta(\lambda), \Delta(\mu))$ are not well understood.

In \cite{CL} and \cite{CP} nonzero homomorphisms were constructed for pairs of partitions that differ by a multiple of a positive root under some conditions on $p$. Another of the few general results is the existence of nonzero homomorphisms for neighboring Weyl modules \cite{Jan}. Homomorphisms between hooks that differ by a positive root where classified in \cite{OM}. More generally homomorphisms corresponding to pairs of partitions that differ by a positive root were classified in \cite{Ku2}, see also \cite{An}, as a corollary of  their result on the corresponding integral $\Ext^1$ group. A result for $GL_3(k)$ was obtained in \cite{BF} when the partitions differ by a multiple of a positive root. In \cite{CoP} all homomorphisms between Weyl modules for $GL_3(k)$ were classified for $p>2$.

In this paper, we classify all homomorphisms between Weyl modules when one of $\lambda, \mu$ is a hook in the sense that we give equivalent arithmetic conditions on $p, \lambda, \mu$ expressed via certain binomial coefficients, such that there is a non-zero map $\phi: \Delta(\lambda) \to \Delta(\mu)$. Moreover, under these conditions we show that $\phi$ is unique up to scalar multiple and provide an explicit description. 

After the first draft of this paper was written, we became aware of \cite{Lou}, where homomorphisms between Specht modules were classified when one of the partitions is a hook, using Khovanov-Lauda-Roquier algebras. By a result of \cite{CL}, this implies a classification of homomorphisms between the corresponding Weyl modules when $p>2$. Our proof is substantially different and covers the case $p=2$. It relies on a detailed analysis of the images of the relations of $\Delta(\lambda)$ under the generators of $\Hom(D(\lambda), \Delta (h)))$, where $h$ is a hook partition and $D(\lambda)=D(\lambda_1) \otimes \cdots \otimes D(\lambda_m)  $ is the indicated tensor product of divided powers of the natural $GL_n(k)$-module, $\lambda = (\lambda_1,...,\lambda_m)$. Working with weight subspaces of $\Delta(h)$, this boils down to a certain linear system, essentially defined over the integers, and we find equivalent conditions so that it has a nonzero solution modulo $p$. A discussion on the relationship with Loubert's result is contained in subsection 3.2.

In Section 7 we consider $G$ a classical group of type  $B_n, C_n, D_n$ and observe that the natural embedding $G \subseteq GL_N(k)$, where $N=2n+1,2n$, induces, under suitable conditions, an injective map 
\[ \Hom_{GL_N(k)}(\nabla(\la), \nabla(\mu)) \to \Hom_{G}(\nabla_G(\overline{\la}), \nabla_G(\overline{\mu})), \]
where $\la, \mu$ are any partitions of $r$ with at most $n$ parts and $\nabla(\la), \nabla_G(\overline{\la})$ are the induced modules of $GL_N(k), G$ respectively with highest weight determined by the partition $\la$ in the usual manner (Proposition 7.2). It follows immediately that our main theorem yields a non-vanishing result for $\Hom_{G}(\nabla_G(\overline{\la}), \nabla_G(\overline{\mu}))$ when $\mu$ is a hook (Corollary 7.1). 

Section 2 is devoted to notation and recollections. The main result, Theorem 3.1, is stated in Section 3. In Section 4 we find equivalent conditions for a map $D(\lambda) \to \Delta(h)$ to induce a map $\Delta(\lambda) \to \Delta(h)$. That these are equivalent to the conditions of Theorem 3.1 is shown in Section 5. As a corollary of the main result, we recover in Section 6 with a different proof the integral $\Ext^1$ groups between any two hooks \cite{MS}. In Section 7 we consider classical groups $G$ of type $B_n, C_n $ and $D_n$ and show the non-vanishing result on induced $G$-modules mentioned above.

\section{Notation and Recollections}
\subsection{Notation.} Let $k$ be an infinite field of
characteristic $p>0$. We will be working with homogeneous polynomial representations of $GL_n(k)$ of degree $r$, or equivalently, with modules over the Schur algebra $S=S_k(n,r)$. A standard reference here is  \cite{Gr}.
Let $V=k^n$ be the natural $GL_n(k)$=module. By $DV=\sum_{i\geq 0}D_iV$ and $\La V=\sum_{i\geq 0}\La^{i}V$
we denote the divided power algebra of $V$ and the exterior algebra of $V$ respectively. We will usually omit $V$ and write $D_i$ and $\La^i$.  By $\wedge(n,r)$ we denote the set of sequences $a=(a_1, \dots, a_n)$ of nonnegative integers that sum to $r$ and by $\wedge^+(n,r)$ we denote the subset of $\wedge(n,r)$ consisting of sequences $\lambda=(\lambda_1, \dots, \lambda_n)$ such that $\lambda_1 \ge \lambda_2 \dots \ge \lambda_n$. Elements of $\wedge^+(n,r)$ are referred to as partitions. A hook $h$ is a partition of the form $h=(a,1^b)$. 

For $\lambda \in \wedge^+(n,r)$, we denote by $\Delta(\lambda)$ the corresponding Weyl module for $S$. For example, if $\lambda =(r)$, then $\Delta(\lambda) =D_r$, and if $\lambda =(1^r)$, then $\Delta(\lambda) =\La^{r}$.

For $a=(a_1,\dots, a_n) \in \wedge(n,r)$, we denote by $D(a)$ or $D(a_1,\dots,a_n)$ the tensor product $D_{a_1}\otimes \dots \otimes D_{a_n}$. All tensor products in this paper are over $k$ unless otherwise specified.
\subsection{Relations for Weyl modules, straightening.} We recall from \cite{ABW} the following description of $\Delta (\lambda)$ in terms of generators and relations. Let $\lambda=(\lambda_1,\dots,\lambda_m) \in \wedge^+(n,r)$, where $\lambda_m >0$. There is an exact sequence of $S$-modules \[
\sum_{i=1}^{m-1}\sum_{t=1}^{\lambda_{i+1}}D(\lambda_1,\dots,\lambda_i+t,\lambda_{i+1}-t,\dots,\lambda_m) \xrightarrow{\square} D(\lambda) \xrightarrow{d'_\lambda} \Delta(\lambda) \to 0,
\]
where the restriction of $ \square $ to the summand $M(t)=D(\lambda_1,\dots,\lambda_i+t,\lambda_{i+1}-t,\dots,\lambda_m)$ is the composition
\[
M(t) \xrightarrow{1\otimes\cdots \otimes \Delta \otimes \cdots 1}D(\lambda_1,\dots,\lambda_i,t,\lambda_{i+1}-t,\dots,\lambda_m)\xrightarrow{1\otimes\cdots \otimes m \otimes \cdots 1} D(\lambda),
\]
where $\Delta:D(\lambda_i+t) \to D(\lambda_i,t)$ and $m:D(t,\lambda_{i+1}-t) \to D(\lambda_{i+1})$ are the indicated components of the comultiplication and multiplication respectively of the Hopf algebra $DV$ and $d'_\lambda$ is the map in \cite{ABW}, Def.II.13.

We recall the straightening law and the standard basis theorem for $\Delta(h)$, where $h=(a,1^b)$. This can be found more generally for skew shapes in \cite{ABW}. Fix an ordered basis $e_1,..,e_n$ of $V$. For simplicity, we denote the element $e_
i$ by $i$ and accordingly the element $e_{i_1}^{(a_1)}  ...  e_{i_t}^{(a_t)} \otimes e_{j_1} \otimes ... \otimes e_{j_b} \in D_a \otimes D(1^b)$ by ${i_1}^{(a_1)}  ...  {i_t}^{(a_t)} \otimes {j_1}\otimes ...\otimes {j_b}$. The image of this element under $d'_h$ will be denoted by ${i_1}^{(a_1)}  ...  {i_t}^{(a_t)} / {j_1} ... {j_b}$. Such an element will be called a tableau of $\Delta(h)$. We note that such a tableau is symmetric in the ${i_1}^{(a_1)}, ...,{i_t}^{(a_t)}$ and  skew-symmetric in the $j_1,\dots,j_b$. Now suppose $i_1<i_2<...<i_t$ and $j_1 \le i_1$. Then in $\Delta(h)$ we have the straightening law
$${i_1}^{(a_1)}...{i_t}^{(a_t)} / {j_1} ... {j_b} = \begin{cases} -\sum\limits_{s\geq 2}{i_1}^{(a_1+1)}  ...{i_s}^{(a_s-1)}...{i_t}^{(a_t)} / {i_s}{j_2} ... {j_b}, & \mbox{if}\;j_1=i_1 \\ -\sum\limits_{s\geq 1}j_1{i_1}^{(a_1)}  ...{i_s}^{(a_s-1)}...{i_t}^{(a_t)} / {i_s}{j_2} ... {j_b}, & \mbox {if}\;j_1<i_1. \end{cases} $$

A $k$-basis of $\Delta(h)$ is the set of the elements ${i_1}^{(a_1)}  ...  {i_t}^{(a_t)} / {j_1} ... {j_b} $, where $a_1+...+a_t=a, i_1<...<i_t$ and $i_1 <j_1< ... <j_b.$

We refer to these basis elements as standard tableaux of $\Delta(h)$. The content (or weight) $\beta$ of such an element ${i_1}^{(a_1)}  ...  {i_t}^{(a_t)} / {j_1} ... {j_b} $ is the sequence $\beta=(\beta_1,...,\beta_n)$, where $\beta_i$ is the number of occurrences of $ i $. If $\beta_{k+1}=...=\beta_n=0$ we may write $\beta=(\beta_1,...,\beta_k)$
\subsection{Weight spaces} If $(a_1,...,a_n) \in \we(n,r)$ and $h \in \we^+(n,r)$ is a hook, we identify the $k$-space $\Hom_S(D(a_1,...,a_n),\Delta(h))$ with the subspace of $\Delta(h)$ spanned by the standard tableaux of content $(a_1,...,a_n)$ according to \cite{AB}, Section 2. Hence from the results in 2.2 we obtain the following. Suppose $\lambda, h \in \we^+(n,r)$, where $\lambda=(\lambda_1,...,\lambda_m), \lambda_m>0, h=(a,1^b)$. If $b \ge 1$, then a $k$-basis of $\Hom_S(D(\lambda),\Delta(h))$ is the set $\{\phi_J:J \in B(\lambda,h)\}$, where 
\[
B(\lambda,h)=\{J=(j_1,...,j_b) \in \mathbb{N}^b : 2 \le j_1 <...<j_b \le m\},
\] and $\phi_J:D(\lambda) \to \Delta(h)$ is the composition 
\begin{align*}D(\lambda) &\to D(\lambda_1,...\lambda_{j-1},1,...,\lambda_{j_b}-1,1,...,\lambda_m) \\
&\simeq D(\lambda_1,...\lambda_{j-1},...,\lambda_{j_b},1,...,\lambda_m)\otimes D(1^b) =D(h) \to \Delta(h),
\end{align*}
where the first map on the factor $D(\lambda_{j_s})$ is the two-fold comultiplication $D(\lambda_{j_s}) \to D(\lambda_{j_s}-1,1), s=1,...,b,$ and on the rest of the factors is the identity map, the middle map is the indicated rearrangement of factors in the tensor product, and the last map is $d'_h$. If $b=0$, then $\Hom_S(D(\lambda),\Delta(h))$ is one dimensional spanned by the map $\phi_{\emptyset}:D(\lambda)\to \Delta(h)$ which is the multiplication $D(\lambda_1,...,\lambda_m) \to D(r)=\Delta(h).$
\section{Main result}
\subsection{Main result}
Suppose $\lambda,h \in \we^+(n,r), \lambda=(\lambda_1,...,\lambda_m), \lambda_m \neq 0$ and $h=(a,1^b)$, where it is understood that $h=(r)$ if $b=0$. We may assume that $m \ge b+1$ since otherwise $\Hom_S(\Delta(\lambda),\Delta(h))=0$ as $\lambda$ is not less than or equal to $h$ in the dominance order of partitions. If $m=b+1$, then the first column of $\lambda$ and the first column of $h$ have equal lengths and by column removal (\cite{Ku1}, Proposition 1.2), we have $\Hom_S(\Delta(\lambda),\Delta(h))=\Hom_{S'}(\Delta(\lambda'),\Delta(h'))$, where $\lambda'=(\lambda_1-1,...,\lambda_m-1), h'=(a-1), S'=S_k(n,r-b-1)$. If $\lambda'$ has only one row, then $\lambda'=h'$ and $\Hom_{S'}(\Delta(\lambda'),\Delta(h'))=k$ by \cite{Jan}, Corollary II.6.24. Therefore we assume throughout that $m \ge b+2$.

In order to state the main result of this paper we will need further notation. If $x,y$ are positive integers, let \[R(x,y)=\gcd\{\tbinom{x}{1},  \tbinom{x+1}{2},...,\tbinom{x+y-1}{y}\}.\]
If $\lambda \in \we^+(n,r)$ such that $\lambda \neq (1,...,1)$, let \[q=\max\{i:\lambda_i \ge 2\}.\] If $\lambda=(1,...,1)$, we define $q=0$ and $\lambda_q=0$.

\begin{theo} Let $k$ be an infinite field of characteristic $p>0$ and $\lambda,h \in \we^{+}(n,r)$, such that $\lambda=(\lambda_1,...,\lambda_m), \lambda_m \neq 0, h=(a,1^b)$ and $m \ge b+2$. Then the following hold.
	\begin{enumerate}
		\item[i)] $\dim\Hom_S(\Delta(\lambda),\Delta(h)) \le 1.$
		\item [ii)] Suppose $q \ge b+1$. Then $\Hom_S(\Delta(\lambda),\Delta(h)) \ne 0$ if and only if $ p $ divides  all of the following integers \begin{align*}
		&R(\lambda_i,\lambda_{i+1}-1), i=1,...,b-1, \\ &R(\lambda_b,\lambda_{b+1}), \\ &R(\lambda_i+1,\lambda_{i+1}), i=b+1,...,m-1.
		\end{align*}
		\item[iii)] Suppose $q \le b$. Then $\Hom_S(\Delta(\lambda),\Delta(h)) \ne 0$ if and only if $ p $ divides  all of the following integers \begin{align*}
		&R(\lambda_i,\lambda_{i+1}-1), i=1,...,q-1, \\ &\lambda_q+b+2-q, \\ &2, \text{if } m-b \ge3.
		\end{align*}
		\item [iv)] In cases ii) and iii) above, a nonzero element of $\Hom_S(\Delta(\lambda),\Delta(h))$ is induced by the map $\psi= \sum_{I \in B(\lambda,h)}d_I \phi_I$, where \[
		d_I=\prod_{u=1}^{b} \prod_{v=u}^{i_u-2} (-\lambda_{v+1}), I=(i_1,...,i_b).\]
	\end{enumerate}
\end{theo}

\noindent A few remarks are in order.\\(1) In the above statement, case $ ii) $, it is understood that $R(\lambda_i, \lambda_{i+1}-1)=0$ if $q<2$. Also in $iv)$, if $u<i_{u}-2$, the corresponding product $\prod_{v=u}^{i_u-2}(-\lambda_{v+1})$ (empty product) is 1. In particular, if $I=(2,3,...,b+1)$, then $d_I=1.$\\
(2) In case $iii)$ we note that the only possible value of $p$ is 2 if $m-b \ge 3$.\\
(3) Consider  $\Hom_S(\Delta(h), \Delta(\lambda))$, where $\lambda,h \in \we^{+}(n,r)$ and $h$ is a hook. We explain here, by a well known argument (for example see \cite{Ku2}, pg. 517-518, for more details ), how Theorem 3.1 may be used for the computation of this Hom space. If $m \ge r$, then by \cite{AB}, Theorem 7.7, and contravariant duality, loc. cit,. Section 2, we have 
\begin{equation*}\Hom_{S_k(m,r)}(\Delta(h), \Delta(\lambda))=\Hom_{S_k(m,r)}(\Delta(\lambda^{'}), \Delta(h^{'})),\end{equation*} where the transpose of a partition $\mu$ is denoted by $\mu^{'}$. Since the transpose of a hook is a hook, the right hand side may be determined by Theorem 3.1 and the reductions stated before it. On the other hand, by  \cite{Gr}, (6.5g), for all $m \ge n$ we have $$\Hom_{S_k(n,r)}(\Delta(h), \Delta(\lambda))=\Hom_{S_k(m,r)}(\Delta(h), \Delta(\lambda)).$$ Hence, given $\lambda,h \in \we^{+}(n,r)$, we may choose $m$ large enough to apply the above. In fact, $m=r$ suffices.
\begin{ex} Carter and Payne \cite{CP} have shown the existence of nonzero homomorphisms between Weyl modules of highest weights that differ by a multiple of a positive root, under suitable conditions on $p$. The next example shows that not all homomorphisms of Theorem 3.1 are compositions of Carter-Payne homomorphisms.

Let $p=2, \la=(c,1^d)$ and $h=(c+2,1^{d-2})$, where $c \ge 2$ is an even integer and $d \ge3$ is an odd integer. With the notation of Theorem  3.1 iii) we have $q=1, m=d+1, b=d-2,$ and $\la_q+b+2-q=c+(d-2)+2+2-1$ is even. Thus $$\Hom_S(\Delta(\la), \Delta(h)) \neq 0.$$ The only partitions $\mu$ such that $\mu-\la$ is a multiple of a positive root are $\mu=(c+1,1^{d-1}), (c,2,1^{d-2})$. For the first of these, Theorem 3.1 iii) yields 
$$\Hom_S(\Delta(\la),\Delta(c+1,1^{d-1}))=0$$
because here $\lambda_1+(d-1)+2-1=c+d$ is odd. For the second, using row removal and Theorem 3.1 iii)  respectively we have
\begin{equation*}
\Hom_S(\Delta(\la),\Delta(c,2,1^{d-2}))=\Hom_{S^{'}}(\Delta(1^d),\Delta(2,1^{d-2}))=0
\end{equation*}
because $d$ is odd, where $S^{'}=S_k(n,d)$.

Hence for $p=2$ and the above choices of $\la, h$, no nonzero homomorphism $\Delta(\la) \to \Delta(h)$ is a composition of Carter-Payne homomorphisms.
\end{ex}

\subsection{Relationship with a result of Loubert}
As  mentioned in the Introduction, Loubert has classified in \cite{Lou} all homomorphisms between Specht modules over the full Khovanov-Lauda-Rouquier algebra when one of the two partitions is a hook. This classification is expressed via certain binomial coefficients and is valid for $e>2$, where $e$ is the least integer so that $1+q+ \cdots q^{e-1}=0$ and $q$ is the deformation parameter. Hence a corresponding classification result is obtained for Specht modules over the group algebra of the symmetric group $\mathfrak{S}_r$. By a result of \cite{CL}, Theorem 3.7, this implies a classification result for homomorphisms between Weyl modules for the general linear group when one of the two partitions is a hook and is valid when $p>2$. The sets of binomial coefficients used by Loubert are different from ours. The purpose of this subsection is to show directly the equivalence of the two results for the general linear group when $p>2$, using elementary divisibility properties of binomial coefficients.

Following \cite{Lou}, if $a=(a_1,...,a_N)$ is a partition of positive integers, the Garnir content $Gc(a) \in k$ of $a$ is defined by \[ Gc(a)=\gcd\left\{\binom{a_i}{j} :  1 \le j \le a_{i+1}-1, 1 \le i \le N-1 \right\},\] with the convention that $\gcd(\emptyset) = 0$. For a partition $\lambda$ of $r$, let $S^{\lambda}$ be the corresponding Specht module for the symmetric group $\mathfrak{S}_r$. The classification result of \cite{Lou} for the group algebra of $\mathfrak{S}_r$ is the following.
\begin{theo}[\cite{Lou}]
Let $k$ be an infinite field of characteristic $p>2$ and $\lambda,h \in \we^{+}(n,r),$ such that $\lambda=(\lambda_1,...,\lambda_m), \lambda_m \neq 0, h=(a,1^b)$ and $m \ge b+1$. Then $\dim \operatorname{Hom}_{\mathfrak{S}_r}(S^{\lambda}, S^{h}) \le 1$ and is equal to 1 if and only if one of the following conditions holds:

(i) There exist $n' \in \{1,...,b+1\}, a \in (\mathbb{Z}_{>0})^{n'}$, and $0 \le m' < p$ such that $Gc(a)=0$ and \[\lambda = (a_1p,...,a_{{n'}-1}p, a_{n'}p-m', 1^{b-{n'}+1} ).\]

(ii) $p$ divides $r$ and there exist $n' \in \{1,...,b\}, a \in (\mathbb{Z}_{>0})^{n'}$, and $0 \le m' < p$ such that $Gc(a)=0$ and \[\lambda = (a_1p,...,a_{n'-1}p, a_{n'}p-m',1^{b-{n'}+2} ).\]

(iii) There exist $n' > b+1, a \in (\mathbb{Z}_{>0})^{n'}$, and $0 \le m' < p$ such that $Gc(a)=0$ and \[\lambda = (a_1p,...,a_bp, a_{b+1}p-1,...,a_{n'-1}p-1, a_{n'}p-1-m').\]
\end{theo}
We note that in case (i) of Loubert's result, both partitions $\lambda$ and $h$ have $b+1$ rows and thus we may apply column removal to $\dim\Hom_S(\Delta(\lambda),\Delta(h))$, which was done at the beginning of subsection 3.1. We intend to show directly the equivalence of conditions (ii) (respectivley, (iii)) of Theorem 3.3 with the conditions iii) (respectively, ii)) of Theorem 3.1 when $p>2$. For this we will need some elementary results concerning binomial coefficients that we now describe.

If $y$ is a positive integer, we denote by $l_p(y)$ the least integer $i$ such that $p^i > y$. For further use, we note that if $y \ge2$ and $m'$ is an integer such that $0 \le m' < p$, then \begin{equation} l_p(yp-1-m') = 1+l_p(y-1).\end{equation}
Indeed, if $p^i > y-1,$ then $p^{i+1}>yp-p$ and thus $p^{i+1} \ge yp> yp-1-m'$. Conversely, if $p^{i+1} > yp-1-m'$, then $p^{i} >  y-\frac{1+m'}{p} \ge y-1.$
\begin{lm} Let $x \ge y$ be positive integers. Then the following hold.
	\begin{enumerate}
		\item[i)] $\gcd \left \{ \binom{x}{1}, \binom{x+1}{2}, ..., \binom{x+y-1}{y}\right \} = \gcd \left \{ \binom{x}{1}, \binom{x}{2}, ..., \binom{x}{y}\right \}.$
		\item[ii)] $p$ divides the above $\gcd$ if and only $p^{l_p(y)}$ divides $x.$
		\item[iii)] Suppose $y \ge2$ and $r$ is an integer such that $0 \le r <p$. Then $p$ divides $\gcd\left \{ \binom{xp}{1}, \binom{xp}{2}, ... \binom{xp}{yp-1-r} \right \}$ if and only if $p$ divides $\gcd\left \{ \binom{x}{1}, \binom{x}{2}, ... , \binom{x}{y-1} \right \}$.
	\end{enumerate}
		\begin{proof} Part  i)  of the lemma follows from Vandermonde's identity \[\binom{x+c-1}{c}=\binom{x}{c} + \sum_{i<c}\binom{x}{i} \binom{c-1}{c-i},\] for $c=2,...,y$.
		
Part ii) is Corollary 22.5 of \cite{Jam} and part iii) follows from part ii) and equation (1). \end{proof}

\end{lm}

Suppose conditions (ii) of Theorem 3.3 hold, so that \begin{equation}\lambda = (\lambda_1, ..., \lambda_{b+2})=(a_1p,...,a_{n'-1}p, a_{n'}p-m',1^{b-{n'}+2} ).\end{equation} We will show that conditions iii) of Theorem 3.1 hold. 

First we note that the number of rows of $\lambda$ is $m=b+2$ and hence the third condition in Theorem 3.1 iii) is empty.

Recall the definition $q=\max\{i:\lambda_i \ge 2\}$ stated before Theorem 3.1. We thus  have two possibilities, 
$q=n'-1$ or $q=n'$.

Let $q=n'-1$. Hence $a_{n'}p-m'=1$ and thus $a_{n'}=1$. From Gc($ a $)=0 we obtain that $p$ divides each of \[ \binom{a_i}{1},\binom{a_i}{2},...,\binom{a_i}{a_{i+1}-1} \]
for $i=1,...,n'-2$. (For $i=n'-1$, there are no corresponding binomial coefficients in Gc(a) since $a_{n'}=1.$) By Lemma 3.4 iii) for $r=0$ we have that $p$ divides each of 
\[\binom{a_ip}{1},\binom{a_ip}{2},...,\binom{a_ip}{a_{i+1}p-1}, \]
that is, each of 
\[\binom{\lambda_i}{1},\binom{\lambda_i}{2},...,\binom{\lambda_i}{\lambda_{i+1}-1}. \]
By Lemma 3.4 i), we obtain that $p$ divides each of 

\[\binom{\lambda_i}{1},\binom{\lambda_i+1}{2},...,\binom{\lambda_i+\lambda_{i+1}-2} {\lambda_{i+1}-1}, \]
which is the condition $p$ divides
$R(\lambda_i,\lambda_{i+1}-1),$ for $i=1,...,q-1,$
of Theorem 3.1 iii).

We need to show the remaining condition of Theorem 3.1 iii), namely that $p$ divides $\lambda_q+b+2-q$. By assumption and equation (2) we have that $p$ divides \[r=a_1p+...+a_{n'-1}p+a_{n'}p-m'+b-n'+2\] and hence $p$ divides \begin{align*}a_{n'-1}p+(a_{n'}p-m')+b-n'+2 =  \lambda_q +b-q+2.\end{align*}

Let $q=n'$. Exactly as before we obtain from $Gc($a$)=0$ that $p$ divides
$R(\lambda_i,\lambda_{i+1}-1),$ for all $i=1,...,q-2$. However, now the assumption $Gc($a$)=0$ contains the additional statement that $p$ divides each of \[\binom{a_{q-1}}{1},\binom{a_{q-1}}{2},...,\binom{a_{q-1}}{a_{q}-1}. \]
Using Lemma 3.4 iii) for $r=m'$, this is equivalent to $p$ divides each of 
\[\binom{a_{q-1}p}{1},\binom{a_{q-1}p}{2},...,\binom{a_{q-1}p}{a_{q}p-1-m'}, \]
which by Lemma 3.4 i) is equivalent to $p$ divides $R(\lambda_{q-1},\lambda_{q}-1).$

Finally we need to show that  $p$ divides $\lambda_q+b+2-q$. This is clear since from (2) we have that $p$ divides \[r=a_1p+...+a_{q-1}p+\lambda_q+b-q+2.\] 

Conversely, suppose conditions iii) of Theorem 3.1 hold. Recall that we are assuming $p>2$, and thus from the third of the previous conditions we have $m=b+2$. Since $p$ divides each of $\lambda_1, ..., \lambda_{q-1}, \lambda_{q}+b+2-q =\lambda_q+m-q $, we conclude that $p$ divides their sum which is equal to $r$. This is the first condition of Theorem 3.3 ii). 

We let $n'=q$ and we define $m'$ by the requirements $0 \le m' <p$ and $\lambda_q \equiv -m \mod p$. Then we have \[\lambda = (a_1p,...,a_{n'-1}p, a_{n'}p-m',1^{b-n'+2}),\]
where $a_1 = \lambda_1/p,...,a_{n'-1} =\lambda_{n'-1}/p$ and $a_{n'}=(\lambda_{n'}+m')/p.$ 

It remains to be shown that Gc($a$)$=0$, where $a=(a_1,...,a_{n'})$. This follows from Lemma 3.4 i) and iii)  since we are assuming that $p$ divides each $R(\lambda_i, \lambda_{i+1}-1)$, for $i=1,...,q-1$.

We have shown the equivalence of conditions in Theorem 3.1 iii) and Theorem 3.3 (ii) under the assumption  $p>2$. The proof of the equivalence of the conditions in Theorem 3.1 ii) and Theorem 3.3 (iii)  under the same assumption is similar and thus omitted.

\section{Relations}
According to 2.3, every map of $S$-modules $D(\lambda) \to \Delta(h)$, is of the form $\phi= \sum_{I \in B(\lambda,h)}c_I\phi_I, c_I \in k$. In this section we determine equivalent conditions on $\lambda, h, c_I$ so that $\phi$ induces a map of $S$-modules $\Delta(\lambda) \to \Delta(h)$.
\begin{lm} Let $\lambda,h \in \we^{+}(n,r),$ where $\lambda=(\lambda_1,...,\lambda_m), \lambda_m \neq 0, h=(a,1^b)$, and let $\phi= \sum_{I \in B(\lambda,h)}c_I\phi_I$. Then $\phi$ induces a map of $S$ modules $\Delta(\lambda) \to \Delta(h)$ if and only if the following relations hold.
\begin{enumerate}
\item[$\mathbf{R_1(t,1):}$]  For $2 \notin I$, $t=1,...,\lambda_2,$
$$\tbinom{\lambda_1+t}{t}c_I+\tbinom{\lambda_1+t-1}{t-1} \sum_{u=1}^{b}(-1)^uc_{I[i_u \to 2]} =0, $$
where $I[i_u \to 2]=(2<i_1 < \cdots <\widehat{i}_u < \cdots <i_b)$ and $\widehat{i}_u$ means that $i_u$ is omitted.

\item[$\mathbf{R_1(t,2):}$]  For $2 \in I$, $t=1,...,\lambda_2-1,$
$$\tbinom{\lambda_1+t-1}{t}c_I=0. $$

\item[$\mathbf{R_i(t,1):}$] For $i \ge 2,  i \in I, i+1 \notin I, t=1,...,\lambda_{i+1},$
$$\tbinom{\lambda_i+t-1}{t}c_I+\tbinom{\lambda_i+t-1}{t-1} c_{I[i \to i+1]} =0, $$
where $I[i \to i+1]=(i_1 < \cdots <i+1 < \cdots <i_b)$ if $I=(i_1 < \cdots <i < \cdots <i_b)$ and $i+1 \notin I.$

\item[$\mathbf{R_i(t,2):}$] For $i \ge 2, i \notin I$, $i+1 \in I$,  $t=1,...,\lambda_{i+1}-1,$
$$\tbinom{\lambda_i+t}{t}c_I=0. $$

\item[$\mathbf{R_i(t,3):}$] For $i \ge 2, i \notin I, i+1 \notin I$,  $t=1,...,\lambda_{i+1},$
$$\tbinom{\lambda_i+t}{t}c_I=0. $$

\item[$\mathbf{R_i(t,4):}$] For $i \ge 2, i \in I, i+1 \in I$,  $t=1,...,\lambda_{i+1}-1,$
$$\tbinom{\lambda_i+t-1}{t}c_I=0. $$
\end{enumerate}
\end{lm}
\begin{proof}

By 2.2, $\phi$ induces a map of $S$-modules $\Delta(\lambda) \to \Delta(h)$ if and only if $\phi(Im\square)=0.$ First we compute the image under $\phi$ of the relations coming from rows 1 and 2 of $\lambda$. Let \[
X_t=1^{(\lambda_1)}\otimes1^{(t)}2^{(\lambda_2-t)}\otimes\cdots\otimes m^{(\lambda_m)} \in Im(\square), \;\; t=1,...,\lambda_2.
\]
Then $\phi(X_t)=A_t+B_t$, where $A_t=\sum_{2 \in J}c_J\phi_J(X_t)$ and $B_t=\sum_{2 \notin J}c_J\phi_J(X_t).$ If $J=(i_1,...,i_b), 1 \le i_s \le n,$ we denote by $T_{t,J} \in \Delta(h)$ the unique row standard tableau of $\Delta(h)$ of weight $(\lambda_1+t,\lambda_2-t,\lambda_3,...,\lambda_m)$ with leg $i_1...i_b$. Then from the definition of $\phi_J$ we obtain
\begin{align}
&A_t=\sum_{2 \in J}c_J \tbinom{\lambda_1+t-1}{t-1}T_{t,J(2 \to 1)}+\sum_{2 \in J} c_J\tbinom{\lambda_1+t}{t} T_{t,J}, \; \; t=1,...,\lambda_2-1,\\&A_t=\sum_{2 \in J}c_J \tbinom{\lambda_1+t-1}{t-1}T_{t,J(2 \to 1)}, \; \; t=\lambda_2,
\end{align}
where $J(2\to v)=(v,j_2,...,j_b)$ if $J=(2<j_2<...<j_b) \in B(\lambda, h)$. Using the straightening law in (3), rearrangement of terms and the identity $-\tbinom{a+b-1}{b-1} + \tbinom{a+b}{b}=\tbinom{a+b-1}{b}$ we have 
\begin{align*}
A_t=&\sum_{2 \in J}c_j \tbinom{\lambda_1+t-1}{t-1}(-\sum_{v \ge 2}T_{t,J(2 \to v)})+\sum_{2 \in J} c_J\tbinom{\lambda_1+t}{t} T_{t,J}\\
=&\sum_{2 \in J}\left( -c_J \tbinom{\lambda_1+t-1}{t-1}+c_J \tbinom{\lambda_1+t}{t}\right)T_{t,J(2 \to 2)} \\ &+\sum_{v \ge 3}\sum_{2 \in J}(-1)c_J\tbinom{\lambda_1+t-1}{t-1}T_{t,J(2 \to v)} \\
=&\sum_{2 \in J}(-c_J) \tbinom{\lambda_1+t-1}{t}T_{t,J}+\sum_{v \ge 3} \sum_{2 \in J}(-c_J)\tbinom{\lambda_1+t-1}{t-1}T_{t,J(2\to v)}.
\end{align*}
Let \[
C=\sum_{v\ge 3} \sum_{2 \in J}(-c_J)\tbinom{\lambda_1+t-1}{t-1}T_{t,J(2\to v)},
\]
and fix $I=(i_1<i_2<...<i_b), \; 3 \le i_1 < ... < i_b \le m.$ Then the coefficient of $T_{t,J}$ in $C$ is equal to \[
-\tbinom{\lambda_1+t-1}{t-1}\sum_{u=1}^{b}(-1)^{u+1}c_{I[i_u \to 2]},
\]
because of skew-symmetry in the leg. Hence
\begin{align}
\nonumber A_t=&\sum_{2 \in J}(-c_J) \tbinom{\lambda_1+t-1}{t}T_{t,J}\\ &-
\tbinom{\lambda_1+t-1}{t-1}\sum_{\substack{I=(i_1<...<i_b) \\3 \le i_1}}\left( \sum_{u=1}^{b} (-1)^{u+1}c_{I[i_u \to 2]}\right)T_{t,I} ,\end{align}
which is valid for $t=1,...,\lambda_2-1.$ For $t=\lambda_2,$ a similar computation from (4) yields 
\begin{align}A_{\lambda_2}=\tbinom{\lambda_1+\lambda_2-1}{\lambda_2-1}\sum_{\substack{I=(i_1<...<i_b) \\3 \le i_1}}\left( \sum_{u=1}^{b} (-1)^{u+1}c_{I[i_u \to 2]}\right)T_{\lambda_2,I}.\end{align}

For $B_t$, the definition of $\phi_J$ yields 
\begin{align}
B_t=\sum_{2 \notin I}c_I \tbinom{\lambda_1+t}{t}T_{t,I}, \; t=1,...,\lambda_2.
\end{align}
From (5) and (7) we have for each $t=1,...,\lambda_2-1,$
\begin{align}
\nonumber \phi(X_t)=&-\sum_{2 \in J}c_J \tbinom{\lambda_1+t-1}{t}T_{t,J}\\ &+
\sum_{2 \notin I}\Big( -\tbinom{\lambda_1+t-1}{t-1}\sum_{u=1}^{b}(-1)^{u+1}  c_{I[i_u \to 2]} + \tbinom{\lambda_1+t}{t}c_I\Big)T_{t,I}
\end{align}
and from (6) and (7)
\begin{align}\phi(X_{\lambda_2})=\sum_{2 \notin I}\Big( -\tbinom{\lambda_1+\lambda_2-1}{\lambda_2-1}\sum_{u=1}^{b}(-1)^{u+1}  c_{I[i_u \to 2]} + \tbinom{\lambda_1+\lambda_2}{\lambda_2}c_I\Big)T_{\lambda_2,I}.
\end{align}
Since all the tableaux appearing in the right hand side of (8) are standard and distinct, and likewise for (9), we conclude that if $\phi$ induces a map of $S$-modules $\Delta(\lambda) \to \Delta(h)$, then the relations $R_1(t,1)$ and $R_1(t,2)$ hold.

Next we compute the image under $\phi$ of the relations coming from rows $i$ and $i+1$ of $\Delta(\lambda)$, $i=2,...,m-1$. Let \[
Y_t=1^{(\lambda_1)}\otimes ... \otimes i^{(\lambda_i)}\otimes i^{(t)}(i+1)^{(\lambda_{i+1}-t)}\otimes ... \otimes m^{(\lambda_m)} \in Im(\square),
\]
$t=1,...,\lambda_{i+1}$. Then $\phi(Y_t)=\sum_{J}c_J\phi_J(Y_t)=A_t+B_t+C_t+D_t,$ where $A_t, B_t, C_t, D_t$ correspond to the summands such that
\begin{align*}
&i \in J,\; i+1 \notin J, \\
&i \notin J,\; i \in J,\\
&i, \; i+1 \notin J, \\
&i, \; i+1 \in J,
\end{align*}
respectively. For $t=1,...,\lambda_i-1$ the definition of $\phi_J$ yields 
\begin{align*}
& A_t=\sum_{\substack{i \in J \\ i+1 \notin J}}c_J \tbinom{\lambda_i+t-1}{t}T_{t,J},\\
&B_t=\sum_{\substack{i \notin J \\ i+1 \notin J}}c_J \tbinom{\lambda_i+t-1}{t-1}T_{t,J[i+1 \to i]} + \sum_{\substack{i \notin J \\ i+1 \in J}}c_J \tbinom{\lambda_i+t}{t}T_{t,J},\\
&C_t=\sum_{\substack{i \notin J \\ i+1 \notin J}}c_J \tbinom{\lambda_i+t}{t}T_{t,J},, \\
&D_t=\sum_{\substack{i \in J \\ i+1 \in J}}c_J \tbinom{\lambda_i+t-1}{t-1}T_{t,J}.
\end{align*}
Noting that $T_{t,J}$, where $i \in J$ and $i+1 \notin J$, appears in both $A_t$ and $B_t$, we obtain
\begin{align*}
\phi(Y_t)=&\sum_{\substack{i \in J \\ i+1 \notin J}}\Big(c_J \tbinom{\lambda_i+t-1}{t}+c_{J[i \to i+1]} \tbinom{\lambda_i+t-1}{t-1}\Big)T_{t,J} \\
&+\sum_{\substack{i \notin J \\ i+1 \in J}}c_J \tbinom{\lambda_i+t}{t}T_{t,J}  +
\sum_{\substack{i \notin J \\ i+1 \notin J}}c_J \tbinom{\lambda_i+t}{t}T_{t,J}  + 
\sum_{\substack{i \in J \\ i+1 \in J}}c_J \tbinom{\lambda_i+t-1}{t-1}T_{t,J},
\end{align*}
valid for $t=1,...,\lambda_{i+1}-1.$ A similar computation for $t=\lambda_{i+1}$ gives 
\begin{align*}
\phi(Y_{\lambda_{i+1}})=&\sum_{\substack{i \in J \\ i+1 \notin J}}\Big(c_J \tbinom{\lambda_i+\lambda_{i+1}-1}{\lambda_{i+1}}+c_{J[i \to i+1]} \tbinom{\lambda_i+\lambda_{i+1}-1}{\lambda_{i+1}-1}\Big)T_{t,J}  \\
&+\sum_{\substack{i \notin J \\ i+1 \notin J}}c_J \tbinom{\lambda_i+\lambda_{i+1}}{\lambda_{i+1}}T_{t,J}.
\end{align*}
Since all the tableaux appearing in the right hand side of the last equation are standard and distinct, and likewise for the penultimate equation, we conclude that if $\phi$ induces a map of $S$-modules $\Delta(\lambda) \to \Delta(h)$, then the relations $R_i(t,j), j=1,2,3,4,$ hold.

Conversely, it is clear from the above that if the relations of the Lemma are satisfied, then $\phi(Im(\square))=0$ and hence $\phi$ induces a map of $S$-modules $\Delta(\lambda) \to \Delta(h).$
\end{proof}

\begin{cor} Suppose $\lambda,h \in \we^{+}(n,r),$ where $h=(a,1^b)$. Then  $$\dim\Hom(\Delta(\lambda), \Delta(h)) \le 1.$$ If $\Hom(\Delta(\lambda), \Delta(h)) \neq 0$, then a nonzero element of  $\Hom(\Delta(\lambda), \Delta(h))$ is the map induced by  
$\psi= \sum_{I \in B(\lambda,h)}d_I\phi_I$, where \[
d_I=\prod_{u=1}^{b} \prod_{v=u}^{i_u-2} (-\lambda_{v+1}),\; I=(i_1,...,i_b).\]

\end{cor}
\begin{proof}
	Suppose $\Hom_S(\Delta(\lambda),\Delta(h) ) \neq 0$, i.e., $\phi=\sum_{I}c_I\phi_I:D(\lambda) \to \Delta(h)$ induces a nonzero map of $S$-modules, $\Delta(\lambda) \to \Delta(h)$. Order the elements of $B(\lambda,h)$ lexicographically. Then $I_0=(2,3,...,b+1)$ is the least element. We will show by induction on this ordering that $c_I=d_I c_{I_{0}}$. The case $I=I_0$ is immediate as $\prod_{v=u}^{i_u-2} (-\lambda_{v+1}) =1$ (empty product). Let $I \in B(\lambda,h), I>I_0$. Then there is an $i+1 \in I, i+1>2$ such $ i \notin I$. Hence there is a $J \in B(\lambda,h)$ such that $I=J[i \to i+1]$. By $R_i(1,1)$ of Lemma 4.1, we have \[c_I=(-\lambda_i)c_J.\] Since $J<I$, we have $c_J=d_Jc_{I_0},$ where $J=(j_1<...<i<...<j_b)$. Hence $c_I=(-\lambda_i)d_Jc_{I_0}=d_Ic_{I_0}.$
	
	It remains to show that, assuming  $\Hom_S(\Delta(\lambda),\Delta(h) ) \neq 0$, the map $\Delta(\lambda) \to \Delta(h)$ induced by $\psi$ is nonzero. This is immediate as the image of the canonical standard tableau
	$d'_\lambda(1^{(\lambda_1)}\otimes \cdots \otimes m^{(\lambda_m)})$ under the induced map is a linear combination of standard tableaux in $\Delta(h)$ that are distinct and the coefficient of the canonical standard tableau in $\Delta(h)$ is 1.
\end{proof}

\section{Equivalence of relations}

\subsection{One direction}Suppose the map $\psi= \sum_{I \in B(\lambda,h)}d_I\phi_I$ defined in Corollary 4.2 induces a map of $S$-modules $\Delta(\lambda) \to \Delta(h)$. We will show in this subsection that the conditions $ii)$ and $iii)$ of Theorem 3.1 are satisfied.

\noindent \textbf{Case 1.} Suppose $q \ge b+1$.

Apply $R_1(t,2)$ for $I=(2,3,...,b,b+1)$. Since $ d_I=1 $, we have $ R(\lambda_1,\lambda_{2}-1)=0 $.

Apply $R_i(t,4)$ for $I=(2,3,...,b,b+1), i=2,...,b-1,$ to obtain 
$R(\lambda_i,\lambda_{i+1}-1)=0, i=2,...,b-1.$

Next we show that $ \lambda_b=\lambda_{b+1}+1=0 $. Indeed, since $ q \ge b+1 $, we have $ \lambda_{b+1} \ge 2 $ and hence may apply $R_b(1,4)$ for $I=(2,3,...,b,b+1)$ to obtain $\lambda_b=0.$ Next, applying $R_1(1,1)$ for $I=(3,4,...,b+2)$ and using the definition of the coefficients $d_J$, we have \[
\pm (\lambda_1+1)\lambda_2\lambda_3...\lambda_{b+1} \pm \lambda_3\lambda_4...\lambda_{b+1} \pm ... +(-1)^b(-\lambda_{b+1})+(-1)^{b+1}=0.
\]
Since $ \lambda_b=0 $, we obtain $ \lambda_{b+1}+1=0 $.

Apply $R_b(t,1)$ for $I=(2,3,...,b,b+2)$ and use the previous result to obtain \[
\tbinom{\lambda_b+t-1}{t}(-\lambda_{b+1})+\tbinom{\lambda_b+t-1}{t}\lambda_b\lambda_{b+1} \Rightarrow
\tbinom{\lambda_b+t-1}{t}=0,\]
$ t=1,2,...,\lambda_{b+1} $. Hence $R(b,b+1)=0$.

It remains to be shown that $R(\lambda_i+1,\lambda_{i+1})=0, i=b+1,...,m-1.$ The cases $i=b+2,...,m-1$ follow from $R_i(t,3)$ for $I=(2,3,...,b+1)$. Finally, apply $R_i(t,1)$ for $i=b+1$ and $I=(2,3,...,b+1)$ to obtain \begin{align*}
0&=\tbinom{\lambda_{b+1}+t-1}{t}+\tbinom{\lambda_{b+1}+t-1}{t-1}(-\lambda_{b+1})  \\ &=\tbinom{\lambda_{b+1}+t-1}{t}+\tbinom{\lambda_{b+1}+t-1}{t-1} =\tbinom{\lambda_{b+1}+t}{t}.
\end{align*}

\noindent \textbf{Case 2.} Suppose $q \le b$.

Apply $R_1(t,2)$ for $I=(2,3,...,b,b+1)$. Since $ d_I=1 $, we have $ R(\lambda_1,\lambda_{2}-1)=0 $.

Apply $R_i(t,4)$ for $I=(2,3,...,b,b+1), i=2,...,b-1,$ to obtain 
$R(\lambda_i,\lambda_{i+1}-1)=0, i=2,...,b-1.$

Since $m \ge b+2$, we may consider $I=(3,4,...,b+2) \in B(\lambda,h)$ and apply $R_1(1,1)$ to obtain \begin{equation}
0=(\lambda_1+1)d_I+\sum_{u=1}^{b}(-1)^ud_{I[i_u \to 2]}.\end{equation} If $i<q$, then $\lambda_i=0$ from what we just proved. Also, if $i>q$, then $\lambda_i=1$ by the definition of $q$. 
\begin{itemize}
	\item Suppose $q \ge 3$.
Using the definition of $d_{I[i_u \to 2]}$ and $d_I$ we have from (10)
\begin{equation*}
0=(\lambda_1+1)d_I+\sum_{u=q-2}^{b}(-1)^ud_{I[i_u \to 2]}=(\lambda_1+1)d_I+(-1)^b\lambda_q+(-1)^b(b+2-q)
\end{equation*}
and $d_I=-\lambda_2d_J=0$, where $J=(2,4,5,...,b+2)$. Hence $\lambda_q+b+2-q=0$.
\item Suppose $q=2$. Then $\lambda_1=0$, $d_I=(-\lambda_2)(-1)...(-1)=(-1)^b\lambda_2$ and $d_{I[i_u \to 2]}=(-1)^{b-u}$. Hence from (10) we have $0=\lambda_2(-1)^b+b(-1)^b.$ Thus again $\lambda_q+b+2-q=0$ is satisfied.
\item Suppose $q=1$. Then (10) yields $0=(\lambda_1+1)(-1)^b+b(-1)^b.$ Thus again $\lambda_q+b+2-q=0$ is satisfied.
\end{itemize}

Finally, let $m-b \ge 3$. Consider $I=(2,3,...,\widehat{q+1}, \widehat{q+2},...,b+3) \in B(\lambda,h)$. Apply $R_{q+1}(1,3)$ to obtain $0=2d_I=2(\pm1).$
  
\subsection{Converse} Suppose conditions ii) or iii) of Theorem 3.1 hold. We will show in this subsection that the relations of Lemma 4.1 are satisfied for the map $\psi= \sum_{I \in B(\lambda,h)}d_I\phi_I$ defined in Corollary 4.2.

\noindent \textbf{Case 1.} Suppose $q \ge b+1$.

\noindent $\mathbf{R_1(t,1)}:$ Let $b=1$. The left hand side of $R_1(t,1)$ is \begin{align*}
\tbinom{\lambda_1+t}{t}d_j-\tbinom{\lambda_1+t-1}{t-1}d_2 &=\tbinom{\lambda_1+t}{t}(-\lambda_2)\cdots(-\lambda_{j-1})-\tbinom{\lambda_1+t-1}{t-1}1\\
&=\tbinom{\lambda_1+t}{t}-\tbinom{\lambda_1+t-1}{t-1}=\tbinom{\lambda_1+t-1}{t}=0,
\end{align*}
where in the second equality we used $R(\lambda_i+1,\lambda_{i+1}), i=2,...,m-1,$ so that $\lambda_i+1=0$, and in the last equality we used $R(\lambda_b,\lambda_{b+1})$ for $b=1$.

Let $b \ge 2$ and $2 \notin I.$ Then $i_1 \ge 3$ and thus $i_{b-1} \ge b+1$ if $I=(i_1,...,i_b)$. Thus $d_I$ is a multiple of $\lambda_b$ which is zero by $R(\lambda_b, \lambda_{b+1}).$ For the same reason, each summand of $\sum_{u=1}^{b}(-1)^{i_u}d_{I[i_u \to 2]}=0$ is zero, except those corresponding to $u=b-1,b.$ But these cancel out as  \[
d_{I[i_{b-1}\to 2]}=(-\lambda_{i_{b-1}})\cdots (-\lambda_{i_b-1})d_{I[i_{b}\to 2]}=1\cdots1d_{I[i_{b}\to 2]}
\]
owing to $R(\lambda_j+1,\lambda_{j+1}), j \ge b+1$. Thus $R_1(t,1)=0$.

\noindent $\mathbf{R_1(t,2)}:$ This is immediate from $R(\lambda_1,\lambda_2-1)$ if $b>1$ and from $R(\lambda_b,\lambda_{b+1})$ if $b=1$.

\noindent $\mathbf{R_i(t,1)}:$ In the left hand side of $R_i(t,1)$  we have $d_{I[i \to i+1]} = -\lambda_id_I$.

Let $i\ge b+1$. Then $-\lambda_i=1$ by $R(\lambda_i+1,\lambda_{i+1})$ and thus \[\tbinom{\lambda_i+t-1}{t}d_I+\tbinom{\lambda_i+t-1}{t-1}d_{I[i\to i+2]}=\tbinom{\lambda_i+t}{t}d_I=0
\]
by $R(\lambda_i+1,\lambda_{i+1})$.

Let $i\le b$. Since $q \ge b+1$, we have $\lambda_{i+1} \ge 2$ meaning that $R(\lambda_i,\lambda_{i+1}-1)$ is nonempty. We have $\lambda_i=0$ by $R(\lambda_i,\lambda_{i+1}-1)$ (if $i<b$) or by $R(\lambda_b,\lambda_{b+1})$ (if $i=b$). Thus the left hand side of $R_i(t,1)$ is $\tbinom{\lambda_i+t-1}{t}d_I$. If $i=b$, then $\tbinom{\lambda_i+t-1}{t}=0$ by $R(\lambda_b,\lambda_{b+1})$. If $i<b$, then $i$ cannot be in the last two positions of $I= (j_1<...<j_b) $ since $j_1 \ge 2$. The condition $i+1 \notin I$ implies that the position in $I$ to the right of $i$ is occupied by an element $\ge i+2$. Hence $d_I$ is a multiple of $\lambda_{i+1}$ by the definition of $d_I$. By $R(\lambda_{i+1},\lambda_{i+2}-1)$ (which is nonempty as $i+2 \le q$), we have $\lambda_{i+1}=0.$

\noindent $\mathbf{R_i(t,2)} $ and $\mathbf{R_i(t,3)} : $  This is immediate from $R(\lambda_i+1,\lambda_{i+1})$ if $i \ge b+1$. If $i\le b$, then since $i \notin I$, the element of $I$ in position $i-1$ is $\ge i+1$. Hence $d_I$ is a multiple of $\lambda_i$. We have $\lambda_i=0$ by $R(\lambda_i,\lambda_{i+1}-1)$ (if $i<b$) or by $R(\lambda_b,\lambda_{b+1})$ (if $i=b$).

\noindent $\mathbf{R_i(t,4)} : $ This is immediate from $R(\lambda_i,\lambda_{i+1}-1)$ if $i<b$ or from $R(\lambda_b,\lambda_{b+1})$ if $i=b$. If $i \ge b+1$, then $i+1 \ge b+2$ and since $i+1 \in I$, we have that $d_I$ is a multiple of $\lambda_i$. But $\lambda_i=0$, by $R(\lambda_i+1,\lambda_{i+1})$.

\noindent \textbf{Case 2.} Suppose $q \le b$. We assume that conditions iii) of Theorem 3.1 hold.
\noindent $\mathbf{R_1(t,2)}:$ This is empty if $q=1$, while for $q>1$ is follows immediately from $R(\lambda_1,\lambda_2-1)$.

\noindent $\mathbf{R_i(t,1)}:$ In the left hand side of $R_i(t,1)$  we have $d_{I[i \to i+1]} = -\lambda_id_I$.

Let $i\ge q$. Then $\lambda_{i+1}=1$, $t=1$ and the left hand side of $R_i(1,1)$ is $\lambda_id_I+d_{I[i\to i+2]}=0.$

Let $i \le q-1 $. Then $\lambda_i=0$ by $R(\lambda_i,\lambda_{i+1}-1)$ and thus it suffices to show that $\tbinom{\lambda_i+t-1}{t}d_I=0$. If $t \neq \lambda_{i+1}$, then $\tbinom{\lambda_i+t-1}{t}=0$ by $R(\lambda_i,\lambda_{i+1}-1)$. Suppose $t=\lambda_{i+1}$. The condition $i+1 \notin I$ implies that the position in $I$ to the right of $i$ is occupied by an element $ \ge i+2$ and  hence $ d_I$ is a multiple of $\lambda_{i+1}$. Now
\[
\tbinom{\lambda_i+\lambda_{i+1}-1}{\lambda_{i+1}}\lambda_{i+1}=(\lambda_i+\lambda_{i+1}-1)\tbinom{\lambda_i+\lambda_{i+1}-2}{\lambda_{i+1}-1} 
\]
which is zero by $R(\lambda_i,\lambda_{i+1}-1)$.

\noindent $\mathbf{R_i(t,2)}:$ If $i \ge q$, then $\lambda_{i+1}=1$ and $R_i(t,2)$ is empty. Let $i <q$. Since  $i \notin I$ and $ i+1 \in I$, we have that $d_I$ is a multiple of $\lambda_{i}$. But as $i < q$, we have $\lambda_i=0$ by $R(\lambda_i,\lambda_{i+1}-1)$.

\noindent $\mathbf{R_i(t,3)}:$ Since $i,i+1 \notin I$ we have $m \ge b+3$. Thus from conditions $iii)$ of Theorem 3.1 we have $ 2=0 $.

If $i >q$, then $ \lambda_i=1,t=1 $ and $R_i(1,3)$ holds.

If $i=q$, then $t=1$. Also, $d_I$ is a multiple of $\lambda_q$ since $q \notin I$ and  $q \le b+1$. Since $(\lambda_q+1)\lambda_q$ is a multiple of 2, we have $(\lambda_q+1)d_I =0$.

If $i<q$, then, as in the previous case,  $d_I$ is a multiple of $\lambda_i$. But $\lambda_i=0$ by $R(\lambda_i, \lambda_{i+1}-1)$.

\noindent $\mathbf{R_i(t,4)}:$ If $i \le q-1$, then $\tbinom{\lambda_i+t-1}{t}=0$ by $R(\lambda_i, \lambda_{i+1}-1)$.

If $i \ge q$, then $R_i(t,4)$ is empty.

\noindent $\mathbf{R_1(t,1)}:$ This is a somewhat lengthy verification, but elementary, as we need to consider several subcases. Suppose $I \in B(\lambda,h), 2 \notin I$ and $t=1,...,\lambda_2$. Recall we are assuming that $q \le b$. For short let $S=\sum_{u=1}^{b}(-1)^ud_{I[i_u \to 2]}$.

\noindent \textbf{Case 1.} Let $m=b+2$. Then $I$ is unique, namely $I=(3,4,...,b+2)$, $d_I=(-1)^b\lambda_2\lambda_3\cdots \lambda_{b+1}$ and 
$S=\sum_{s=3}^{b+2}(-1)^sd_{(2,3,...,\widehat{s},...,b+2)}.$

\begin{itemize}
\item Let $q=1$. Then $\lambda_2=...=\lambda_m=1$ and hence $d_I=(-1)^b$ and $S=(-1)^bb$. We have $ t=1 $ and the left hand side of $R_1(1,1)$ is \[(\lambda_1+1)d_I+S=(-1)^b(\lambda_1+b+1)=0,\]
since $\lambda_q+b-q+2=0$.

\item Let $q=2$. Then $\lambda_3=...=\lambda_m=1$ and hence $d_I=(-1)^b\lambda_2$ and $S=(-1)^bb$. The left hand side of $R_1(t,1)$ is
\[(-1)^b\tbinom{\lambda_1+t}{t}\lambda_2+(-1)^bb\tbinom{\lambda_1+t-1}{t-1}=(-1)^b\tbinom{\lambda_1+t-1}{t}\lambda_2, 
\]
since $\lambda_2+b=0$. If $t<\lambda_2$, then by $R(\lambda_1,\lambda_2-1)$ we have $\tbinom{\lambda_1+t-1}{t}=0$. If $t=\lambda_2$, we have 
\[\tbinom{\lambda_1+\lambda_2-1}{\lambda_2}\lambda_2=(\lambda_1+\lambda_2-1)\tbinom{\lambda_1+\lambda_2-2}{\lambda_2-1}=0
\]
again by $R(\lambda_1,\lambda_2-1)$.

\item Let $q \ge 3$. Then $\lambda_{q+1}=...=\lambda_m=1$ and $d_I=(-1)^b\lambda_2...\lambda_q=0$ because $\lambda_2=0$.
Also, $d_{(2,3,...,\widehat{s},...,b+2)}=0$, $ s=3,...,q-1 $, because $ \lambda_2=...=\lambda_{q+1}=0. $ Hence \begin{align*}S&=\sum_{s=q}^{b+2}(-1)^sd_{(2,3,...,\widehat{s},...,b+2)}\\&=(-1)^b\lambda_q+(-1)^b+...+(-1)^b=(-1)^b(\lambda_q+b-q+2).\end{align*}
Thus $S=0$ and the left hand side of $R_1(t,1)$ is 0.

\end{itemize}

\noindent \textbf{Case 2.} Let $m>b+2$. Then $2=0$  and we will omit the signs in the following computations. Suppose $I=(i_1,...,i_{b+2}) \in B(\lambda,h)$, $2 \notin I$. Hence $i_s \ge s+2 $ for all $s$.
\begin{itemize}
	\item Let $q=1$. Then $\lambda_2=...=\lambda_m=1$ and hence $d_I=1$ and $S=1+...+1=b$. We have $ t=1 $ and the left hand side of $R_1(1,1)$ is $(\lambda_1+1)d_I+S=\lambda_1+1+b=0$
	since $\lambda_q+b-q+2=0$.
	\item Let $q=2$. Then $\lambda_3=...=\lambda_m=1$ and hence $d_I=\lambda_2$ and $S=b$. The left hand side of $R_1(t,1)$ is
	\[\tbinom{\lambda_1+t}{t}\lambda_2+b\tbinom{\lambda_1+t-1}{t-1}=\tbinom{\lambda_1+t-1}{t}\lambda_2, 
	\]
	since $\lambda_2+b=0$. Exactly as in the case $m=b+3, q=2$ one concludes that this is 0. 
	
\item Let $q \ge 3$. Then $\lambda_{q+1}=...=\lambda_m=1$. Also $d_I$ is a multiple of $\lambda_2$ since $i_1 \ge 3$. Thus
 $d_I=0$ because of $R(\lambda_2,\lambda_2-1)$. We will show that $S=0$ and for this we consider cases.
 
 1) Let $i_{q-2}=q$. Then $d_{(2,i_1,...,\widehat{i_s},...,i_b)}=0$, $ s=1,...,q-3 $, because $ \lambda_{q-1}=0$ and at position $q-2$ of $(2,i_1,...,\widehat{i_s},...,i_b)$ there is the index $i_{q-2}=q > q-1$. Also, 
$d_{(2,i_1,...,\widehat{i_{q-2}},...,i_b)}=\lambda_q$	and for every $s=q-1,...,b$, $d_{(2,i_1,...,\widehat{i_s},...,i_b)}=1.$	Hence $S=\lambda_q+b-q+2=0.$

2) Let $i_{q-2} \ge q+1$ and $i_{q-3}=q-1$.Then $d_{(2,i_1,...,\widehat{i_s},...,i_b)}=0$, $ s=1,...,q-3$, and  $d_{(2,i_1,...,\widehat{i_s},...,i_b)}=\lambda_q$, $ s=q-2,...,b$. Hence $S=(b-q+3)\lambda_q$. Since $\lambda_q+b-q+2=0$ and $2=0$, it follows that $S=(\lambda_q+1)\lambda_q=0$.

3) Let $i_{q-3} \ge q$. Then $d_{(2,i_1,...,\widehat{i_s},...,i_b)}=0$ for all $s$ and thus $S=0$.
\end{itemize}

\section{Special case: two hooks}
In this section we work over the integral Schur algebra $S_{\Z} = S_{\mathbb{Z}}(n,r)$, where $n \ge r$. The corresponding Weyl modules will be denoted by $\Delta_{\mathbb{Z}}(\lambda)$. It is well known that if $k$ is an infinite field, then $\Delta(\lambda) = k\otimes_{\mathbb{Z}}\Delta_{\mathbb{Z}}(\lambda).$ Consider hooks $h, h(d) \in \we^+(n,r)$, where $h=(a,1^b), h(d)=(a+d,1^{b-d})$ and $d>0.$ As a corollary of Theorem 3.1, we obtain a different proof of the following result from \cite{MS}.
\begin{cor}
If $d \ge 2$, then \[\Ext^1_{\Z}(\Delta_{\Z}(h), \Delta_{\Z}(h(d)))=
\begin{cases} \Z_2, & \;r+d\;\; \mbox{odd}\\ 0, & \;r+d\;\; \mbox{even.}\end{cases}\]
\end{cor}
\begin{proof} From the short exact sequence of $S_{\Z}$-modules \[
	0 \to \Delta_{\Z}(h(d)) \to D_{\Z}(a+d-1)\otimes_{\Z} \La_{\Z} (b-d+1) \to \Delta_{\Z}(h(d-1)) \to 0
	\]
the long exact sequence in cohomology yields the injective map \[ 
\Ext^1_{S_{\Z}}(\Delta_{\Z}(h), \Delta_{\Z}(h(d))) \to \Ext^1_{S_{\Z}}(\Delta_{\Z}(h), D_{\Z}(a+d-1)\otimes_{\Z} \La_{\Z} (b-d+1)) 
\]
because $\Hom_{S_{\Z}}(\Delta_{\Z}(h), \Delta_{\Z}(h(d))) =0$ as $ h, h(d) $ are distinct partitions. Applying twice Theorem 2 of \cite{Ku1} and contravariant duality, we have
\begin{equation*}\begin{split}
\Ext&^1_{S_{\Z}}(\Delta_{\Z}(h), D_{\Z}(a+d-1)\otimes_{\Z} \La_{\Z} (b-d+1))\\=&\Ext^1_{S_{\Z}}(D_{\Z}(a-1) \otimes_{\Z}, \La_{\Z}(d), D_{\Z}(a+d-1))\\=&\Ext^1_{S_{\Z}}(\La_{\Z}(a+d-1),D_{\Z}(d) \otimes_{\Z}, \La_{\Z}(a-1))\\=&  \Ext^1_{S_{\Z}}(\La_{\Z}(d),D_{\Z}(d)),\end{split}
\end{equation*}
and it is well known that this last extension is $\Z_2$ when $d>1$ (for example, see \cite{A}, Section 4). Hence $\Ext^1_{S_{\Z}}(\Delta_{\Z}(h), \Delta_{\Z}(h(d)))$ is 0 or $\Z_2$. Now from the universal coefficient theorem \cite{AB}, we have \[ \Hom(\Delta(h),\Delta(h(d)))=\textrm{Tor}_1^\Z(k,\Ext^1_{S_{\Z}}(\Delta_{\Z}(h),\Delta_{\Z}(h(d))))
\]
and from Theorem 3.1 iii) we have $\Hom(\Delta(h),\Delta(h(d)))=0$ unless $a+(b-d)+2-1$ is even, in which case $\Hom(\Delta(h),\Delta(h(d)))=k$. Hence the result follows.
\end{proof}

\section{Classical groups}
The purpose of this section is to indicate how Theorem 3.1 yields a non-vanishing result for homomorphisms between Weyl modules (or induced modules) for $G$ a special orthogonal or symplectic group of rank $n$, when one of the partitions is a hook, see Corollary 7.1. For this, we observe a general fact in Proposition 7.2 that relates homomorphisms spaces between induced modules for $GL_N(k)$ and $G$, where $N=2n+1, 2n$.
\subsection{Notation and recollections}We begin by fixing notation and recalling some facts concerning induced modules of reductive groups specialized to the classical groups. For more details see \cite{Jan}, II.2. Let $k$ be an infinite field and $GL_N(k)$ the general linear group of $N \times N$ invertible matrices over $k$. The coordinate ring $k[GL_N(k)]$ is a left $GL_N(k)$-module with action given be right translation, $(g_1f)(g_2)=f(g_2g_1)$, for $f \in k[GL_N(k)], g_1,g_2 \in GL_N(k)$. Let $B \subseteq GL_N(k)$ be the subgroup of lower triangular matrices and  $T \subseteq GL_N(k)$ be the subgroup of triangular matrices. For each $i=1,...,N$, let $\epsilon_i:T_N \to k^{*}$, $k^{*}=k-{0}$, be the function such that $\epsilon_i(diag(t_1,..,t_N))=t_i$.  The character group $X(T)$ is free abelian of the $\epsilon_i$, $X(T)=\{\la_1 \epsilon_1 + \cdots \la_N \epsilon_N : \la_i \in \mathbb{Z} \}$. We have the set of positive roots $\Phi^{+}=\{\epsilon_i-\epsilon_j: 1\le i <j \le N\}$ and a partial order on $X(T)$: $\la \le \mu$ if $\mu-\la$ is a sum of positive roots.

As is customary, we will identify the $N$-tuple of integers  $(\la_1, \cdots, \la_N)$ with $\la_1 \epsilon_1 + \cdots + \la_N\epsilon_N$. Each $\la \in X(T)$ extends uniquely to a character of $B$ which we denote again by $\la$. Let $k_{\la}$ be the corresponding one dimensional $B$-module. By $\nabla(\la)$ we denote the induced module $\{f \in k[GL_N(k)]: f(bg)=\la(b)f(g) \; \hbox{for all} \; g \in GL_N(k), b \in B\}.$ This module is nonzero if and only if $\la=(\la_1, \cdots, \la_N)$ satisfies $ \la_1 \ge \la_2 \ge \cdots \ge \la_N$. Moreover, if $\nabla(\la) \neq 0$, then it is a highest weight module of weight $\la$ and it has simple socle $L(\la)$ of highest weight $\la$. The corresponding Weyl module $\Delta(\la)$ for $ GL_N(k) $ is the contravariant dual of $\nabla(\la)$.

We have the set $\wedge^{+}(N,r)$ of partitions of $r$ with at most $N$ parts. The transpose $\la^{'}=(\la^{'}_1, \cdots, \la^{'}_q)$ of a partition $\la=(\la_1, \cdots, \la_N)$ is defined by $\la^{'}_j = \#\{i:\la_i \ge j\}.$ 

Next we consider special orthogonal and symplectic groups. In the case of special orthogonal groups, we will assume that $p>2$. Let 
\begin{align*}SO_{N}(k)&=\{X \in SL_{N}(k): X^tJ_{N}X=J_{N}\}, \\
Sp_{2n}(k)&=\{X \in SL_{2n}(k): X^tJ^{'}_{2n}X=J^{'}_{2n}\} 
\end{align*}
where
\[J_{N}=\begin{pmatrix} 0&&1\\&\iddots& \\1& &0 \end{pmatrix}, \; J^{'}_{2n}=\begin{pmatrix} 0&J_n\\-J_n&0\end{pmatrix}.\]
For $G=SO_{2n+1}(k), Sp_{2n}(k), SO_{2n}(k)$, let $T_G$ and $B_G$ be the subgroups  of diagonal and lower triangular matrices respectively in $G$. We will denote the restriction of $\epsilon_i$ to $T_G$ by $\overline{\epsilon}_i$. The character group of $G$, $X(T_G)=\{\lambda_1 \overline{\epsilon}_1 + \cdots + \lambda_n \overline{\epsilon}_n : \lambda_i \in \mathbb{Z} \},$
is free abelian on the $\overline{\epsilon}_i$. Each $\overline{\la} \in X(T_G)$ extends uniquely to a character of $B_G$ which we denote again by $\overline{\la}$. We have the set of positive roots in each case
\begin{align*}SO_{2n+1}(k)&:  \Phi^{+}=\{\overline{\epsilon}_i \pm \overline{\epsilon}_j:1 \le i < j \le n \} \cup \{\overline{\epsilon}_i: 1 \le i \le n\},\\
Sp_{2n}(k)&:  \Phi^{+}=\{\overline{\epsilon}_i \pm \overline{\epsilon}_j:1 \le i < j \le n \} \cup \{2\overline{\epsilon}_i: 1 \le i \le n\},\\
SO_{2n}(k)&:  \Phi^{+}=\{\overline{\epsilon}_i\pm \overline{\epsilon}_j:1 \le i < j \le n \},
\end{align*}
and a partial order on $X(T_G)$: $\overline{\la }\le \overline{\mu}$ if $\overline{\mu}-\overline{\la}$ is a sum of positive roots.

Let $k_{\overline{\la}}$ be the corresponding one dimensional $B_G$-module. The coordinate ring $k[G]$ of $G$ is a left $G$-module with action given be right translation. By $\nabla_G(\overline{\la})$ we denote the induced module $\{f \in k[G]: f(bg)=\overline{\la}(b)f(g) \; \hbox{for all} \; g \in G, b \in B_G\}.$ The corresponding Weyl module $\Delta_G(\overline{\la})$ for $ G $ is the contravariant dual of $\nabla_G(\overline{\la})$.

For $G=SO_{2n+1}(k), Sp_{2n}(k)$ we have $\nabla_G(\overline{\la}) \neq 0$ if and only if $\la_1 \ge \cdots \la_n \ge0$ and for $G=SO_{2n}(k)$ we have $\nabla_G(\la) \neq 0$ if and only if $\la_1 \ge \cdots \la_{n-1} \ge  |\la_n| \ge0$, where $\overline{\la}=\la_1 \overline{\epsilon}_1+\cdots + \la_n \overline{\epsilon}_n$.  Moreover, if $\nabla_G(\overline{\la}) \neq 0$, then it is a highest weight module of weight $\overline{\la}$ and it has simple socle $L_G(\overline{\la})$ of highest weight $\overline{\la}$.

Let $\la=(\la_1,...,\la_n) \in \wedge^{+}(n,r)$ and $N=2n+1,2n$. According to the above, we have the induced $GL_N{(k)}$-module $\nabla{(\la)}$ of highest weight $\la_1\epsilon_1+...+\la_n \epsilon_n$ and also the induced $G$-module $\nabla_G{(\overline{\la})}$ of highest weight $\la_1\overline{\epsilon}_1+...+\la_n \overline{\epsilon}_n$.

%
%
%
%
\subsection{Homomorphisms between induced modules for classical groups}Let $G$ be one of $SO_{2n+1}(k), Sp_{2n}(k),$ $ SO_{2n}(k)$ and let $\la, \mu \in \we^{+}(n,r)$, where $\la=(\la_1,...,\la_n)$, $\mu=(\mu_1,...,\mu_n)$. Consider the following assumptions. 
\begin{enumerate}
	\item[(a)] $p>2$ if $G=SO_{2n+1}(k), SO_{2n}(k)$.
	\item[(b)] $\la_n = \mu_n=0$ if $G=SO_{2n}(k)$.
\end{enumerate}
We intend to show in this section the following non-vanishing result.

\begin{cor}
Let $\la, \mu \in \we^+(n,r)$, where $\mu =h$ is a hook. If $\la, \mu$ and $p$ satisfy the assumptions $ (a), (b) $ above and $p$ satisfies the divisibility properties of (ii) or (iii) of Theorem 3.1, then \[\Hom_G(\nabla_G(\overline{\la}), \nabla_G(\overline{\mu})) \neq 0.\]
\end{cor}

The inclusion  $G \subseteq GL_N(k)$, where $N=2n+1,2n$, yields, via restrictions of polynomial functions, a homomorphism of coordinate rings $k[GL_N(k)] \to k[G].$ It was shown by Donkin \cite{Do1}, Proposition 1.4, that $\psi$ induces a surjection $\nabla(\la) \to \nabla_G(\overline{\la})$ of $G$-modules when $G=Sp_{2n}(k)$ and $\la \in \we^+(n,r)$. The same argument shows the corresponding result for $G=SO_{2n+1}(k), SO_{2n}(k)$ provided assumptions (a), (b) above hold for $\la$, \cite{Ma2} Proposition 1.3.

Corollary 7.1 is an immediate consequence of Theorem 3.1, the well known isomorphism $$\Hom_{GL_N(k)}(\Delta(\la), \Delta(\mu)) \simeq \Hom_{GL_N(k)}(\nabla(\la), \nabla(\mu)),$$ see for example \cite{AB}, and the following general result which is independent of the rest of the paper. We were not able to locate a reference for it in the literature.
\begin{prop}
	Let $\la, \mu \in \we^{+}(n,r)$. Under the assumptions (a) and (b) above, the restriction map $k[GL_N(k)] \to k[G]$ induces an injective map \[ \Hom_{GL_N(k)}(\nabla(\la), \nabla(\mu)) \to \Hom_{G}(\nabla_G(\overline{\la}), \nabla_G(\overline{\mu})). \]
\end{prop}
\begin{proof} (a) If $\alpha =(\alpha_1,\cdots,\alpha_s)$ is a sequence of nonnegative integers, let $\Lambda(\alpha)=\Lambda^{\alpha_1}V \otimes \cdots \otimes\Lambda^{\alpha_s}V,$ a tensor product over $k$ of exterior powers of the natural $GL_N(k)$-module $V$ of column vectors, and $S(\alpha)=S_{\alpha_1}V \otimes \cdots \otimes S_{\alpha_s}V$
	a tensor product of symmetric powers of $V$. For a partition $\la$ we denote by $L_{\la}V$ the Schur module of \cite{ABW}, Definition II.1.3, and recall that it is defined as the image of a particular $GL_N(k)$-map $d_{\la}:\Lambda(\lambda) \to S(\lambda^{'})$, where $\lambda^{'}$ denotes the transpose partition of $\la$. As $GL_N(k)$-modules, $\nabla(\lambda)\simeq L_{\la^{'}}V$ according to  \cite{Do2}. Thus we have the a surjective map of $G$-modules $$\pi_{\la^{'}}:L_{\la^{'}}V \to \nabla_G(\overline{\la})$$
	induced by the surjective map $\nabla(\la) \to \nabla_G(\overline{\la})$ mentioned before the proposition. We describe next the kernel of $\pi_{\la^{'}}.$
	
	(a1) Suppose $G=SO_{2n+1}(k)$ or $SO_{2n}(k)$. Let $N=2n+1$ or $2n$. For $G=SO_{2n+1}(k)$ we fix an ordered basis $$\{y_1< \dots< y_n< y_{n+1}< y_{\overline{n}}<\dots<y_{\overline{1}}\},$$
	of $V$, where $\overline{i} = 2n+2-i$, such that $<y_i,y_{\overline{i}}>=<y_{\overline{i}},y_i>=1$, for $i=1,...,n$, $<y_{n+1},y_{n+1}>=1$ and the inner product between any other basis elements is zero. For $G=SO_{2n}{(k)}$, we have a similar situation where we simply drop the element $y_{n+1}$. (Here $\overline{i} = 2n+1-i$, $i=1,...,n$). According to \cite{Ma2}, under assumptions (a) and (b) of (7.2) for $\la$, the kernel of $\pi_{\la^{'}}$ is generated over $k$ by all elements of the form   \begin{equation} 
	\sum_{1 \le i_1 < \cdots i_t \le N}d_{\la^{'}}(w_1 \otimes \cdots \otimes y_{i_1}\cdots y_{i_t}w_u\otimes \cdots \otimes y_{\overline{i_1}}\cdots y_{\overline{i_t}}w_v \otimes \cdots \otimes w_q),
	\end{equation}
	where $m \le min\{\la^{'}_u,\la^{'}_v\}$, $w_u \in \Lambda^{\la^{'}_u-m}V$, $w_v \in \Lambda^{\la^{'}_v-m}V$, $w_i \in \Lambda^{\la_i^{'}}V,$ ($i \neq u,v$), and $\lambda^{'}=(\la_1^{'},...,\la_q^{'})$.
	
	(a2) For $G=Sp_{2n}(k)$, we fix an ordered basis $$\{y_1< ... < y_n < y_{\overline{n}}<\dots<y_{\overline{1}}\},$$
of $V$, where $\overline{i} = 2n+1-i$, such that $<y_i,y_{\overline{i}}>=-<y_{\overline{i}},y_i>=1$, for $i=1,...,n$, and the inner product between any other basis elements is zero. From \cite{Ma1}, we know that under assumptions (a) and (b) of (7.2) for $\la$, the kernel of $\pi_{\la^{'}}$
 is generated over $k$ by all elements of the form 
\begin{equation} 
\sum_{1 \le i_1 < \cdots i_t \le N}d_{\la^{'}}(y_{i_1}y_{\overline{i_1}}\cdots y_{i_t}y_{\overline{i_t}}w_1 \otimes w_2 \otimes \cdots \otimes w_q),
\end{equation}
where $2t \le \la^{'}_1$, $w_1 \in \Lambda^{\la^{'}_1-2t}V$, $w_i \in \Lambda^{\la^{'}_i}V$, ($i \neq 1$) and $\lambda^{'}=(\la_1^{'},...,\la_q^{'})$.

Now the main point is that from (11) and (12) if follows that every $G$-weight $\overline{\nu}=\nu_1\overline{\epsilon}_1+...+\nu_n\overline{\epsilon}_n$ of $ker\pi_{\la^{'}}$ satisfies
\begin{equation}\nu_1+...+\nu_n \le r-2.\end{equation}
In particular, this means that $\overline{\mu}$ is not a weight of $ker\pi_{\la^{'}}$. Since $\nabla_G(\overline{\mu})$ has simple socle of highest weight $\overline{\mu}$, we conclude that $\Hom_G(ker\pi_{{\la}'},\nabla_G(\overline{\mu}))=0$. Hence, if $\phi \in \Hom_{GL_N{(k)}}(L_{\la^{'}}V,L_{\mu^{'}}V)$, then the restriction of the composition 

\begin{equation}L_{\la^{'}}V \xrightarrow{\phi} L_{\mu^{'}}V \xrightarrow{\pi_{\mu^{'}}} \nabla_G(\overline{\mu})  \end{equation}
to $ker\pi_{{\la}'}$ is zero and thus we obtain a map of $G$-modules $\nabla_G({\la} )\xrightarrow{\overline{\phi}} \nabla_G({\mu})$.

(b) It remains to be shown that if $\phi \in \Hom_{GL_N{(k)}}(L_{\la^{'}}V,L_{\mu^{'}}V)$ is nonzero, then $\overline{\phi }$ is nonzero. Indeed, if $\phi \neq 0$, then since $L_{\mu^{'}}V$ has simple socle of highest weight $\mu=\mu_1\epsilon_1+...+\mu_n \epsilon_n$, the image $Im\phi$ contains a vector $v$ of $GL_N(k)$-weight $\mu$. But $v$ is a vector of $G$-weight $\overline{\mu}=\mu_1\overline{\epsilon}_1+...+\mu_n \overline{\epsilon}_n$. Since $\mu_1+\cdots +\mu_n=r$, we conclude from (13) for $\mu$ in place of $\la$ that $\overline{\mu}$ is not a weight of $ker\pi_{\mu^{'}}$. Thus the image of the composition (14) contains a vector of $G$-weight $\overline{\mu}$ and hence is nonzero.

\end{proof}
	
\section*{Acknowledgments} We thank H. Geranios for bringing \cite{Lou} to our attention and the reviewers for constructive comments and suggestions that helped improve the presentation of the paper.

\end{document}